\documentclass[a4paper,11pt,twoside]{article}
\usepackage{makeidx}
\usepackage{amsfonts}
\usepackage{fullpage}
\usepackage{amsmath}
\usepackage{amssymb}
\usepackage{amsthm}
\usepackage{appendix}
\usepackage{manfnt}
\usepackage{marginnote}
\usepackage{xfrac}
\newtheorem{thm}{Theorem}[section]
\newtheorem{lem}[thm]{Lemma}
\newtheorem{prop}[thm]{Proposition}
\newtheorem{cor}[thm]{Corollary}
\theoremstyle{definition}
\newtheorem{ex}[thm]{Example}
\newtheorem{defn}[thm]{Definition}

\theoremstyle{remark}
\newtheorem{rmk}[thm]{Remark}
\usepackage{theoremref}
\usepackage[all]{xy}
\usepackage{mathrsfs}
\newcommand{\op}{\mathrm{op}}

\newcommand{\fp}{\mathrm{fp}}

\newcommand{\Coker}{\mathrm{Coker}}
\newcommand{\Ker}{\mathrm{Ker}}
\newcommand{\Img}{\mathrm{Im}}

\newcommand{\Ab}{\mathrm{Ab}}

\newcommand{\End}{\mathrm{End}}
\newcommand{\Ext}{\mathrm{Ext}}

\newcommand{\rmods}{\mathrm{mod}\text{-}}
\newcommand{\lmods}{\text{-}\mathrm{mod}}
\newcommand{\Rmods}{\mathrm{Mod}\text{-}}
\newcommand{\Lmods}{\text{-}\mathrm{Mod}}
\newcommand{\pp}{\mathrm{pp}}
\newcommand{\Hom}{\mathrm{Hom}}
\newcommand{\defeq}{\overset{\mathrm{def}}{=}}

\newcommand{\Ses}{\mathrm{Ses}}
\newcommand{\A}{\mathcal{A}}
\newcommand{\C}{\mathcal{C}}

\newcommand{\D}{\mathcal{D}}

\newcommand{\LD}{\mathrm{L}}

\newcommand{\PP}{\mathcal{P}}
\newcommand{\Hg}{\mathrm{H}}
\newcommand{\Ann}{\mathrm{Ann}}
\newcommand{\U}{\mathcal{U}}
\newcommand{\V}{\mathcal{V}}
\newcommand{\Ell}{\mathcal{L}}
\begin{document}
\title{\textbf{The defect recollement, the MacPherson-Vilonen construction, and pp formulas}}
\author{Samuel Dean}
\date{}
\maketitle
\abstract{For any abelian category $\A$, Auslander constructed a localisation $w:\fp(\A^\op,\Ab)\to \A$ called the defect, which is the left adjoint to the Yoneda embedding $Y:\A\to\fp(\A^\op,\Ab)$. If $\A$ has enough projectives, then this localisation is part of a recollement called the defect recollement. We show that this recollement is an instance of the MacPherson-Vilonen construction if and only if $\A$ is hereditary. We also discuss several subcategories of $\fp(\A^\op,\Ab)$ which arise as canonical features of the defect recollement, and characterise them by properties of their projective presentations and their orthogonality with other subcategories. We apply some parts of the defect recollement to the model theory of modules. Let $R$ be a ring and let $\phi/\psi$ be a pp-pair. When $R$ is an artin algebra, we show that there is a smallest pp formula $\rho$ such that $\psi\leqslant\rho\leqslant\phi$ which agrees with $\phi$ on injectives, and that there is a largest pp formula $\mu$ such that $\psi\leqslant \mu\leqslant \phi$ and $\psi R=\mu R$. When $R$ is left coherent, we show that there is a largest pp formula $\sigma$ such that $\psi\leqslant\sigma\leqslant \phi$ which agrees with $\psi$ on injectives, and that the pp-pair $\psi/\phi$ is isomorphic to a pp formula if and only if $\psi=\sigma$, and that there is a smallest pp formula $\nu$ such that $\psi\leqslant \nu\leqslant\phi$ and $\phi R=\nu R$. We also show that, for any pp-pair $\phi/\psi$, $w(\phi/\psi)\cong (D\psi)R/(D\phi)R$, where $D$ is the elementary duality of pp formulas. We also give expressions for $w(\phi/\psi)$ in terms of free realisation of $\phi$ and $\psi$.}
\tableofcontents
\section{Introduction: The defect recollement}
\emph{Throughout this paper, $\A$ denotes an abelian category with enough projectives.}

A functor $F:\A^\op\to\Ab$ is \textbf{finitely presented} if there is an exact sequence
\begin{displaymath}
\xymatrix{\Hom_\A(-,B)\ar[r]&\Hom_\A(-,C)\ar[r]& F\ar[r]& 0}
\end{displaymath}
for objects $B,C\in\A$. We write $\fp(\A^\op,\Ab)$ for the category of finitely presented functors $\A^\op\to\Ab$.

For categories $\C$ and $\D$, a \textbf{localisation} $G:\C\to\D$ is a functor which preserves finite limits and has a fully faithful right adjoint. It is well-known that for any such functor, if $\C$ is abelian then $\D$ is also abelian and $G$ is a Serre quotient functor which expresses the fact that $\D$ is the Serre quotient $\D=\C/\Ker G$.

Note that the projective objects in $\fp(\A^\op,\Ab)$ are precisely the representable functors. In \cite{auslander1965}, Auslander showed that $\fp(\A^\op,\Ab)$ is an abelian category of global projective dimension $0$ or $2$ with kernels and cokernels computed object-wise. He also showed that the Yoneda embedding $Y:\A\to\fp(\A^\op,\Ab)$ has an exact left adjoint
$w:\fp(\A^\op,\Ab)\to\A$ which we call the \textbf{defect}, following Russell in \cite{russell}. In particular, $w$ is a localisation, and we obtain what is often referred to as Auslander's formula,
\begin{displaymath}
\A\simeq\fp(\A^\op,\Ab)/\fp_0(\A^\op,\Ab),
\end{displaymath}
where $\fp_0(\A^\op,\Ab)=\Ker(w)$. Below is Auslander's description of $\fp_0(\A^\op,\Ab)$.
\begin{thm}\emph{\cite[Proposition 3.2]{auslander1965}}\thlabel{auszero} For any $F\in\fp(\A^\op,\Ab)$, the following are equivalent.
\begin{enumerate}
\item $F\in\fp_0(\A^\op,\Ab)$.
\item Any morphism $g:B\to C\in \A$ which gives a projective presentation 
\begin{displaymath}
\xymatrix{\Hom_\A(-,B)\ar[r]^{g_*}&\Hom_\A(-,C)\ar[r]&F\ar[r]&0}
\end{displaymath}
of $F$ is an epimorphism in $\A$.
\item There is an epimorphism $g:B\to C\in \A$ which gives a projective presentation 
\begin{displaymath}
\xymatrix{\Hom_\A(-,B)\ar[r]^{g_*}&\Hom_\A(-,C)\ar[r]&F\ar[r]&0}
\end{displaymath}
of $F$.
\item $(F,G)=\Ext^1(F,G)=0$ for any left exact functor $G:\A^\op\to\Ab$.
\item $(F,G)=0$ for any left exact functor $G:\A^\op\to\Ab$.
\item $(F,\Hom_\A(-,X))=0$ for any $X\in\A$.
\end{enumerate}
\end{thm}
We note the following corollary which is useful for our case, in which $\A$ has enough projectives.
\begin{cor}\thlabel{wproj}The following hold.
\begin{enumerate}\item For any $F\in\fp(\A^\op,\Ab)$, $F\in\fp_0(\A^\op,\Ab)$ if and only if $FP=0$ for any projective $P\in\A$.
\item For any sequence of morphisms
\begin{displaymath}
\xymatrix{F\ar[r]^{\alpha}&G\ar[r]^{\beta}&H}
\end{displaymath}
in $\fp(\A^\op,\Ab)$ such that $\beta\alpha=0$,
\begin{displaymath}
\xymatrix{wF\ar[r]^{w\alpha}&wG\ar[r]^{w\beta}&wH}
\end{displaymath}
is exact if and only if
\begin{displaymath}
\xymatrix{FP\ar[r]^{\alpha_P}&GP\ar[r]^{\beta_P}&HP}
\end{displaymath}
is exact for any projective $P\in\A$.
\item For any morphism $\alpha:F\to G$ in $\fp(\A^\op,\Ab)$, $w\alpha:wF\to wG$ is an epimorphism if and only if $\alpha_P:FP\to GP$ is an epimorphism for any projective $P\in\A$.
\item For any morphism $\beta:G\to H$ in $\fp(\A^\op,\Ab)$, $w\beta:wG\to wH$ is a monomorphism if and only if $\beta_P:GP\to HP$ is a monomorphism for any projective $P\in\A$.
\end{enumerate}
\end{cor}
\begin{proof}
There is a morphism $g:B\to C\in \A$ which gives a projective presentation 
\begin{displaymath}
\xymatrix{\Hom_\A(-,B)\ar[r]^{g_*}&\Hom_\A(-,C)\ar[r]&F\ar[r]&0}
\end{displaymath}
of $F$. Since $\A$ has enough projectives, $g$ is an epimorphism if and only if $\Hom_\A(P,g):\Hom_\A(P,B)\to \Hom_\A(P,B)$ is an epimorphism for any projective $P\in\A$. Therefore, item 1 follows from \thref{auszero}.

Since $w$ is exact, $w(\Ker\beta/\Img\alpha)\cong\Ker(w\beta)/\Img(w\alpha)$, so item 2 follows by applying item 1 to $\Ker\beta/\Img\alpha$. 

Items 3 and 4 follow from item 2 by setting $H=0$ and $F=0$ respectively.
\end{proof}

Now we will discuss how Auslander's localisation can be seen as part of a larger structure known as a recollement.

\begin{defn}\thlabel{recoldef}A \textbf{recollement of abelian categories} is a situation consisting of additive functors 
\begin{displaymath}
\xymatrix{\mathcal{C}'\ar[rr]|{i_*}&&\mathcal{C}\ar@<2.5ex>[ll]^{i^!}\ar@<-2.5ex>[ll]_{i^*}\ar[rr]|{j^*}&&\mathcal{C}''\ar@<2.5ex>[ll]^{j_*}\ar@<-2.5ex>[ll]_{j_!}}\end{displaymath}
between abelian categories $\mathcal{C}'$, $\mathcal{C}$ and $\mathcal{C}''$ such that the following hold:
\begin{list}{$\bullet$}
\item $\Img(i_*)=\Ker (j^*)$.\item $i_*$ is fully faithful and $i^*\dashv i_*\dashv i^!$.
\item $j_!\dashv j^*\dashv j_*$ and $j_!$ and $j_*$ are fully faithful.
\end{list}
If each of the categories $\C'$, $\C$ and $\C''$ appearing in a recollement have enough projectives, then we say the recollement is a \textbf{recollement situation with enough projectives}. In such a case, we follow Franjou and Pirashvili \cite{comprecoll} by saying that the recollement is \textbf{pre-hereditary} if $\LD_2(i^*)(i_*P)=0$ for each projective $P\in\C'$. 
\end{defn}
\begin{defn}For a ring $R$, we write $R\Lmods$ for the category of left $R$-modules, $\Rmods R$ for the category of right $R$-modules, $R\lmods$ for the category of finitely presented left $R$-modules, and $\rmods R$ for the category of finitely presented right $R$-modules.
\end{defn}
\begin{ex}Let $R$ be an artin algebra and let $M$ be a right $R$-module. The abelian group $\Hom_R(M,R)$ is a left $R$-module with $(rh)x=r(hx)$ for any $h\in \Hom_R(M,R)$, $r\in R$ and $x\in M$. Consider the morphism  
\begin{displaymath}
f:M\otimes_R \Hom_R(M,R)\to R:h\otimes_R x\mapsto hx.
\end{displaymath}
Its image
\begin{displaymath}
\Img f=\sum_{h\in\Hom_R(M,R)}\Img h
\end{displaymath}
is a two-sided ideal of $R$, called the \textbf{trace ideal of $M$}.

Using the notation of \thref{recoldef}, Geigle and Lenzing show at \cite[Proposition 5.3]{gl} that the only recollements of abelian categories with $\C=\rmods R$ are those of the form 
\begin{displaymath}
\xymatrix{\rmods R/\mathfrak{a}\ar[rr]|{i_*}&&\rmods R\ar@<2.5ex>[ll]^{i^!}\ar@<-2.5ex>[ll]_{i^*}\ar[rr]|{j^*}&&\rmods \End_R(P)\ar@<2.5ex>[ll]^{j_*}\ar@<-2.5ex>[ll]_{j_!}}\end{displaymath}
where 
\begin{align*}
j_*&=\Hom_R(P,-)|_{\rmods R}\cong -\otimes_R \Hom_R(P,R)|_{\rmods R}
\\j_!&=-\otimes_{\End_R(P)}P|_{\rmods \End_R(P)}
\\j_*&=\Hom_{\End_R(P)}(\Hom_R(P,R),-)|_{\rmods \End_R(P)}
\end{align*}
and $\mathfrak{a}$ is the trace ideal of $P$, for some finitely generated projective right $R$-module $P$. The fully faithful functor $i_*$ is just the inclusion of $\rmods R/\mathfrak{a}$ as the full subcategory of $\rmods R$ consisting of those finitely generated right $R$-modules $M$ such that $M\mathfrak{a}=0$. Geigle and Lenzing also show that every localising subcategory is of this form for some finitely generated projective $P$, and is also colocalising, thus giving rise to a recollement of this form.
\end{ex}

In \cite{dean3}, the author and Russell have shown that, if $\A$ has enough projectives, then there is a recollement
\begin{displaymath}
\xymatrix{\fp_0(\A^\op,\Ab)\ar[rr]|{\subseteq}&&\fp(\A^\op,\Ab)\ar@<2.5ex>[ll]^{(-)_0}\ar@<-2.5ex>[ll]_{(-)^0}\ar[rr]|{w}&&\mathcal{A}\ar@<2.5ex>[ll]^{Y}\ar@<-2.5ex>[ll]_{\LD_0(Y)}}
\end{displaymath}
where $Y$ is the Yoneda embedding and $\LD_0(Y)$ is the unique right exact functor which agrees with $Y$ on projective objects, known as the \textbf{zeroeth left derived functor of $Y$}. We refer to this as the \textbf{defect recollement}.

The following lemma, along with its dual, gives a recipe for recovering the left hand side of a recollement of abelian categories from the right hand side.
\begin{lem}\thlabel{counitstuff}\emph{\cite[Section 4.2]{comprecoll}} In the notation of \thref{recoldef}, for any recollement of abelian categories, there is an exact sequence of functors
\begin{displaymath}
\xymatrix{j_!j^*\ar[r]^\zeta&1_\C\ar[r]&i_*i^*\ar[r]&0,}
\end{displaymath}
where $\zeta:j_!j^*\to 1_\C$ is the counit of the adjunction $j_!\dashv j^*$ and $i^*j_!=0$.

For any recollement situation with enough projectives, $\LD_1(i^*)j_!=0$ and if $i^*C=0$ for an object $C\in\C$ then there is an isomorphism
\begin{displaymath}
\Ker(\zeta_C)\cong i_*\LD_1(i^*)C
\end{displaymath}
which is natural in $C$.
\end{lem}
\begin{proof}
Since $j_!$ is fully faithful, $j^*\zeta$ is an isomorphism, and therefore $j^*\Coker(\zeta)=0$. Since $j^*\Coker(\zeta)=0$, $\Img(i_*)=\Ker(j^*)$ and $i_*$ is fully faithful, there is some functor $i^?:\C\to\C'$ such that $\Coker(\zeta)\cong i_*i^?$, so there is an exact sequence
\begin{displaymath}
\xymatrix{j_!j^*\ar[r]^\zeta&1_\C\ar[r]&i_*i^?\ar[r]&0.}
\end{displaymath}
We will show that $i^?\cong i^*$. For any $C\in \C$ and $C'\in\C'$, the induced sequence
\begin{displaymath}
\xymatrix{0\ar[r]&\Hom_{\C}(i_*i^?C,i_*C')\ar[r]&\Hom_{\C}(C,i_*C')\ar[r]&\Hom_\C(j_!j^*C,i_*C')}
\end{displaymath}
is exact, but $\Hom_\C(j_!j^*C,i_*C')\cong \Hom_\C(j^*C,j^*i_*C')=0$ so there is an isomorphism
\begin{displaymath}
\Hom_{\C'}(i^?C,C')\cong \Hom_\C(i_*i^?C,i_*C')\cong \Hom_\C(C,i_*C')
\end{displaymath}
which is natural in $C$ and $C'$. Therefore, $i^?\dashv i_*$, so $i^?\cong i^*$, as required. Since $i^*j_!\dashv j^*i_*=0$, $i^*j_!=0$.

Suppose the recollement situation has enough projectives. We need to prove that $\LD_1(i^*)j_!=0$. Let $C''\in\C''$ be given, along with an epimorphism $P\to C''$ where $P$ is projective. Since $j_!$ is right exact and preserves projectives (because it has an exact right adjoint), there is an epimorphism $j_!\Omega C''\to \Omega j_!C''$ and therefore $i^*\Omega j_!C''=0$. The exact sequence
\begin{displaymath}
\xymatrix{0\ar[r]&\Omega j_!C''\ar[r]&j_!P\ar[r]&j_!C''\ar[r]&0}
\end{displaymath}
induces an exact sequence
\begin{displaymath}
\xymatrix{\LD_1(i^*)j_!P\ar[r]&\LD_1(i^*)j_!C''\ar[r]&i^*\Omega j_!C''.}
\end{displaymath}
Since $j_!P$ is projective, $\LD_1(i^*)j_!P=0$. Also, as shown above $i^*\Omega j_!C''=0$. Therefore $\LD_1(i^*)j_!C''=0$, as required.

Suppose $i^*C=0$. The exact sequence
\begin{displaymath}
\xymatrix{0\ar[r]&\Ker(\zeta_C)\ar[r]&j_!j^*C\ar[r]^\zeta&C\ar[r]&0}
\end{displaymath}
induces another exact sequence
\begin{displaymath}
\xymatrix{\LD_1(i^*)j_!j^*C\ar[r]&\LD_1(i^*)C\ar[r]&i^*\Ker(\zeta_C)\ar[r]&i^*j_!j^*C\ar[r]^\zeta&i^*C\ar[r]&0}
\end{displaymath}
which is natural in $C$.

Since $\LD_1(i^*)j_!j^*C=0$ and $i^*j_!j^*C=0$, there is an isomorphism $\LD_1(i^*)C\cong i^*\Ker(\zeta_C)$ which is natural in $C$. Now, $j^*\zeta$ is an isomorphism because $j_!$ is fully faithful, and therefore $j^*\Ker(\zeta_C)=0$, so by the first part, there are isomorphisms $i_*\LD_1(i^*)C\cong i_*i^*\Ker(\zeta_C)\cong\Ker(\zeta_C)$ which are natural in $C$.
\end{proof}

For any $F\in\fp(\A^\op,\Ab)$, by the dual of \thref{counitstuff}, $F_0$ can be constructed as the kernel of the component $\eta_F:F\to YwF$ of the unit $\eta$ of the adjunction $w\dashv Y$. Since $Y$ is fully faithful, $w\eta_F$ is an isomorphism. Therefore, by \thref{wproj}, $(\eta_F)_P:FP\to \Hom_\A(P,wF)$ is an isomorphism for any projective $P\in\A$, so $F_0P=0$.

Let $A\in\A$ and an epimorphism $\pi:P\to A$ be given, where $P$ is projective, and let $F\in\fp(\A^\op,\Ab$) be given. There is a commutative square
\begin{displaymath}
\xymatrix{FA\ar[d]_{F\pi}\ar[r]&\Hom_\A(A,wF)\ar[d]^{\Hom_\A(\pi,wF)}
\\
FP\ar[r]&\Hom_\A(P,wF),}
\end{displaymath}
where the top and bottom sides are given by $(\eta_F)_A$ and $(\eta_F)_P$ respectively. Since the bottom side is an isomorphism and the right hand side is a monomorphism, $F_0A$ is the kernel of $F\pi$. That is,
\begin{displaymath}
F_0A=\{x\in FA:(F\pi)x=0\}.
\end{displaymath}
Therefore, $F_0$ is the subfunctor of $F$ which consists of elements of $F$ which are annihilated by those morphisms which factor through a projective. Also, $F_0$ is the largest subfunctor of $F$ which is in $\fp_0(\A^\op,\Ab)$.

The functor $(-)^0$ is right exact, so is determined by its action on representable functors, which will be discussed in Section \ref{mv}. By \thref{counitstuff}, for any $F\in\fp(\A^\op,\Ab)$, $F^0$ can be constructed as the cokernel of the component $\LD_0(Y)(wF)\to F$ of the counit of the adjunction $\LD_0(Y)\dashv w$.

For the purposes of doing computations with the defect recollement, it is harmless to simplify matters by setting
\begin{align*}
w\Hom_\A(-,A)&=A
\\\text{and }\LD_0(Y)(P)&=\Hom_\A(-,P)
\end{align*}
for any $A\in\A$ and any projective $P\in\A$. When we choose to do this, the components $w\Hom_\A(-,A)\to A$ and $\LD_0(Y)w\Hom_\A(-,P)\to\Hom_\A(-,P)$ of the respective counits can be taken to be identities.

The MacPherson-Vilonen construction of abelian recollements, given below, first appeared in \cite[Section 1]{macvil} in order to simplify the definition of the category of perverse sheaves over a stratified space. In \cite[Theorem 8.7]{comprecoll} (see \thref{mvtest} below) Franjou and Pirashvili give necessary and sufficient conditions for a recollement with enough situation with enough projectives to be given by this construction.
\begin{defn}Let $\C'$ and $\C''$ be abelian categories, let $F:\C''\to\C'$ be a right exact functor, let $G:\C''\to\C'$ be a left exact functor, and let $\alpha:F\to G$ be a natural transformation. The \textbf{MacPherson-Vilonen construction} for $\alpha$ is recollement of abelian categories 
\begin{displaymath}
\xymatrix{\mathcal{C}'\ar[rr]|{i_*}&&\mathcal{C}(\alpha)\ar@<2.5ex>[ll]^{i^!}\ar@<-2.5ex>[ll]_{i^*}\ar[rr]|{j^*}&&\mathcal{C}''\ar@<2.5ex>[ll]^{j_*}\ar@<-2.5ex>[ll]_{j_!}}
\end{displaymath}
given by the following data.
\begin{itemize}
\item The abelian category $\C(\alpha)$, which has objects and morphisms as follows.
\begin{itemize}
\item Objects $(X,V,g,f)$ given by an object $X\in\C''$, an object $V\in\C'$ and morphisms
\begin{displaymath}
\xymatrix{FX\ar[r]^{f}&V\ar[r]^{g}&GX}
\end{displaymath}
such that $gf=\alpha_X$.
\item Morphisms $(x,v):(X,V,g,f)\to(X',V',g',f')$ given by morphisms $x:X\to X'\in\C''$ and $v:V\to V'\in\C'$ such that the diagram
\begin{displaymath}
\xymatrix{FX\ar[d]_{Fx}\ar[r]^{f}&V\ar[d]^v\ar[r]^{g}&GX\ar[d]^{Gx}
\\FX'\ar[r]^{f'}&V'\ar[r]^{g'}&GX'}
\end{displaymath}
commutes.
\end{itemize}
\item The functors $j^*$, $i^*$ and $i^!$ are defined by 
\begin{align*}
j^*(X,V,g,f)&=X
\\i^*(X,V,g,f)&=\Coker f
\\i^!(X,V,g,f)&=\Ker g
\end{align*}
for each $(X,V,g,f)\in\C(\alpha)$.
\item The functor $i_*$ is defined by $i_*V=(0,V,0,0)$ for each $V\in\C'$.
\item The functors $j_!$ and $j_*$ are defined by
\begin{align*}
j_!X&=(X,FX,\alpha_X,1_{FX})\\
j_*X&=(X,GX,1_{GX},\alpha_X)
\end{align*}
for each $X\in\C''$.
\end{itemize}
\end{defn}
\begin{defn}
Let
\begin{displaymath}
\xymatrix{\mathcal{C}'\ar[rr]|{i_*}&&\mathcal{C}\ar@<2.5ex>[ll]^{i^!}\ar@<-2.5ex>[ll]_{i^*}\ar[rr]|{j^*}&&\mathcal{C}''\ar@<2.5ex>[ll]^{j_*}\ar@<-2.5ex>[ll]_{j_!}}
\end{displaymath}
be a recollement of abelian categories. Let $\epsilon:j^*j_*\to 1_{\C''}$ be the counit of the adjunction $j^*\dashv j_*$ and let $\zeta:j_!j^*\to 1_\C$ be the counit of the adjunction $j_!\dashv j^*$. Note that $\epsilon$ is an isomorphism. The \textbf{norm} of the recollement is the natural transformation $N:j_!\to j_*$ given by $N=(\zeta j_*)(j_!\epsilon^{-1})$.
\end{defn}
In the defect recollement, the counit $\epsilon:wY\to 1_\A$ is taken to be the identity, and, at each projective $P\in\A$, the component $\zeta_{\Hom_\A(-,P)}:\LD_0(Y)(P)\to \Hom_\A(-,P)$ of the counit $\zeta:\LD_0(Y)w\to 1_{\fp(\A^\op,\Ab)}$ is also taken to the identity. Therefore, the norm of the defect recollement is the unique natural transformation $N=\zeta Y:\LD_0(Y)\to Y$ which is the identity on projectives. 

In Section \ref{mv} we will apply the following theorem to prove that the defect recollement is an instance of the MacPherson-Vilonen construction if and only if each object of $\A$ has projective dimension at most one.
\begin{thm}\thlabel{mvtest}\emph{\cite[Theorem 8.7]{comprecoll}} A recollement situation with enough projectives,
\begin{displaymath}
\xymatrix{\mathcal{C}'\ar[rr]|{i_*}&&\mathcal{C}\ar@<2.5ex>[ll]^{i^!}\ar@<-2.5ex>[ll]_{i^*}\ar[rr]|{j^*}&&\mathcal{C}'',\ar@<2.5ex>[ll]^{j_*}\ar@<-2.5ex>[ll]_{j_!}}
\end{displaymath}
is an instance of the MacPherson-Vilonen construction if and only if it is pre-hereditary and there is an exact functor $r:\C\to\C'$ such that $ri_*\cong 1_{\C'}$. In particular, if the recollement is pre-hereditary and admits such an exact functor $r$, then it is the MacPherson-Vilonen construction for the natural transformation $rN$, where $N$ is its norm.
\end{thm}
\section{When is the defect recollement MacPherson-Vilonen?}\label{mv}
\begin{defn}We say that $\A$ is \textbf{hereditary} if each of its objects have projective dimension at most one.
\end{defn}
In \cite{dean3}, it was observed that the defect recollement is pre-hereditary if and only if $\A$ is hereditary. The purpose of this section is to show that the defect recollement is MacPherson-Vilonen if and only if $\A$ is hereditary. Before doing that, we need to summarise some background.

For any $A,B\in\A$, we will write $\PP(A,B)$ for the subgroup of $\Hom_\A(A,B)$ consisting of the morphisms $A\to B$ which factor through a projective object, and we will write $\underline{\Hom}_\A(A,B)=\Hom_\A(A,B)/\PP(A,B)$. We will write $\underline{\A}$ for the additive category which has the same objects as $\A$ and has $\Hom_{\underline{\A}}(A,B)=\underline{\Hom}_\A(A,B)$, for each $A,B\in\A$. The category $\underline{\A}$ is a well-known construction known as the \textbf{projective stabilisation} of $\A$. Clearly there is a canonical full functor $\A\to\underline{\A}$. For any object $A\in\A$ we write $\underline{A}$ for its image in $\underline{\A}$. Similarly, we write $\underline{f}:\underline{A}\to\underline{B}$ for the image of a morphism $f:A\to B\in\A$ in $\underline{\A}$.

Let $A\in\A$ be given. We note that $\PP(-,A)\in\fp(\A^\op,\Ab)$ because, for any projective $P\in\A$ and any epimorphism $\pi:P\to A$, the functor $\PP(-,A)$ is the image of the induced morphism $\pi_*:\Hom_\A(-,P)\to\Hom_\A(-,A)$. We also have $\underline{\Hom}_\A(-,A)\in\fp(\A^\op,\Ab)$ since it is the cokernel of the inclusion $\PP(-,A)\to\Hom_\A(-,A)$.

We will make use of the bilinear functor $\Ext^n_\A:\A^\op\times\A\to\Ab$ for integers $n\geqslant 0$. Since we do not assume that $\A$ has enough injectives, it is worth discussing the basic properties of this functor very briefly. For any $A,B\in\A$, $\Ext^n_\A(A,B)$ is defined as $\Hg^n\Hom_\A(P_\cdot,B)$, where $P_\cdot\in\mathrm{K}_{\geqslant 0}(\A)$ is a projective resolution of $A$. Any short exact sequence in $\A$ gives rise to a long exact sequence as usual, and does so in either variable. Since we do assume that $\A$ has enough projectives, $\Ext^n(A,-)\in\fp(\A,\Ab)$ for any $A\in\A$ and any $n\geqslant 0$.

Let $\Ses(\A)$ denote the category of short exact sequences in $\A$, where morphisms are given by homotopy classes of chain maps. Auslander's characterisation (see \thref{auszero}) of $\fp_0(\A^\op,\Ab)$ shows that $\fp_0(\A^\op,\Ab)\simeq\Ses(\A)$, where a short exact sequence
\begin{displaymath}
\xymatrix{0\ar[r]&A\ar[r]^f&B\ar[r]^g\ar[r]&C\ar[r]&0}
\end{displaymath}
corresponds to the finitely presented functor $G$ with projective presentation 
\begin{displaymath}
\xymatrix{0\ar[r]&\Hom_\A(-,A)\ar[r]^{f_*}&\Hom_\A(-,B)\ar[r]^{g_*}&\Hom_\A(-,C)\ar[r]&G\ar[r]&0.}
\end{displaymath}
Since it is obvious that $(\Ses(\A^\op))^\op\simeq \Ses(\A)$, we must have an equivalence
$$(\fp_0(\A,\Ab))^\op\simeq\fp_0(\A^\op,\Ab).$$
We can describe this equivalence more concretely, as follows.
In \cite{dean3}, an equivalence
\begin{displaymath}
W:(\fp_0(\A,\Ab))^\op\to\fp_0(\A^\op,\Ab)
\end{displaymath}
is constructed such that, for each object $A\in\A$, there is an isomorphism 
\begin{displaymath}W\Ext^1_\A(A,-)\cong \underline{\Hom}_\A(-,A)
\end{displaymath}
which is natural in $A$. It is formally defined by $(WF)A=\Ext^2(F,\Hom_\A(A,-))$ for any $F\in\fp_0(\A^\op,\Ab)$ and $A\in\A$. One can easily show that if 
\begin{displaymath}
\xymatrix{0\ar[r]&\Hom_\A(C,-)\ar[r]^{g^*}&\Hom_\A(B,-)\ar[r]^{f^*}&\Hom_\A(A,-)\ar[r]&F\ar[r]&0}
\end{displaymath}
is a projective presentation for $F\in\fp_0(\A,\Ab)$ then
\begin{displaymath}
\xymatrix{0\ar[r]&\Hom_\A(-,A)\ar[r]^{f_*}&\Hom_\A(-,B)\ar[r]^{g_*}&\Hom_\A(-,C)\ar[r]&WF\ar[r]&0}
\end{displaymath} 
is a projective presentation of $WF$, and it is clear from this that $W$ is the equivalence $(\fp_0(\A,\Ab))^\op\simeq \fp_0(\A^\op,\Ab)$ which corresponds to the obvious equivalence $(\Ses(\A^\op))^\op\simeq\Ses(\A)$.

\begin{lem}\thlabel{0hom}\emph{\cite[Lemma 13]{dean3}} For any $A\in\A$, there is an isomorphism $(\Hom_\A(-,A))^0=\underline{\Hom}_\A(-,A)$ which is natural in $A$.
\end{lem}
\begin{proof}
For any $F\in\fp(\A^\op,\Ab)$, $F^0$ is constructed as the cokernel of the component $\zeta_F:\LD_0(Y)(wF)\to F$ of the counit of the adjunction $\LD_0(Y)\dashv w$. Choose a projective presentation 
\begin{displaymath}
\xymatrix{P_1\ar[r]&P_0\ar[r]&A\ar[r]&0}
\end{displaymath}
of $A$. By the naturality of the counit there is a commutative diagram
\begin{displaymath}
\xymatrix{\Hom_\A(-,P_1)\ar[r]&\Hom_\A(-,P_0)\ar@{=}[d]\ar[r]&(\LD_0Y)A\ar[d]^{\zeta_{\Hom_\A(-,A)}}\ar[r]&0\\
&\Hom_\A(-,P_0)\ar[r]&\Hom_\A(-,A)&}
\end{displaymath}
with the top row exact. Therefore the image of $\zeta_{\Hom_\A(-,A)}$ is the image of the induced morphism $\Hom_\A(-,P_0)\to\Hom_\A(-,A)$, which is $\PP(-,A)$. Therefore, by the construction of $(\Hom_\A(-,A))^0$,
\begin{displaymath}
(\Hom_\A(-,A))^0=\Hom_\A(-,A)/\PP(-,A)=\underline{\Hom}_\A(-,A).
\end{displaymath}
\end{proof}
\begin{cor}\thlabel{stabproj}For any $A\in\A$, $\underline{\Hom}_\A(-,A)$ is a projective in $\fp_0(\A^\op,\Ab)$. Furthermore, there are enough projectives in $\fp_0(\A^\op,\Ab)$ of this form.
\end{cor}
\begin{proof}For any functor $F\in\fp_0(\A^\op,\Ab)$ and any object $A\in\A$, there are isomorphisms
\begin{displaymath}
(\underline{\Hom}_\A(-,A),F)\cong ((\Hom_\A(-,A))^0,F)\cong(\Hom_\A(-,A),F)\cong FA
\end{displaymath}
which are natural in $F$ and $A$. Therefore, $\underline{\Hom}_\A(-,A)$ is projective in $\fp_0(\A^\op,\Ab)$ for any $A\in\A$.

Let $F\in \fp(\A^\op,\Ab)$ be given. There is $A\in\A$ and an epimorphism $\Hom_\A(-,A)\to F$. By applying $(-)^0$, which is right exact, it follows from \thref{0hom} that there is an epimorphism $\underline{\Hom}_\A(-,A)\to F^0$. If $F\in\fp_0(\A^\op,\Ab)$ then $F\cong F^0$, so there is an epimorphism $\underline{\Hom}_\A(-,A)\to F$.
\end{proof}
\begin{cor}$\underline{\A}$ has weak kernels. That is, for any $g:B\to C\in\A$, there is a morphism $f:A\to B\in\A$ such that the sequence
\begin{displaymath}
\xymatrix{\underline{\Hom}_\A(-,A)\ar[rr]^{\underline{f}_*} &&\underline{\Hom}_\A(-,B)\ar[rr]^{\underline{g}_*}&&\underline{\Hom}_\A(-,C)}
\end{displaymath}
is exact.
\end{cor}
\begin{proof}
Let $K$ be the kernel of the induced natural transformation $\underline{g}_*$. Since $\fp_0(\A^\op,\Ab)$ is closed under kernels, $K\in\fp_0(\A^\op,\Ab)$. Compose the inclusion $K\to\underline{\Hom}_\A(-,B)$ and an epimorphism $\underline{\Hom}_\A(-,A)\to K$ to find $\underline{f}_*$.
\end{proof}
\begin{lem}\thlabel{1hom}\emph{\cite[Section 4.1.1]{dean3}} For any $A\in\A$, 
\begin{displaymath}
\LD_2((-)^0)(\underline{\Hom}_\A(-,A))\cong\underline{\Hom}_\A(-,\Omega A),
\end{displaymath}
where $\Omega A$ is the first syzygy of $A$.
\end{lem}
\begin{proof}
Let $P\to A$ be an epimorphism, where $P$ is projective. We obtain an exact sequence
\begin{displaymath}
\xymatrix{0\ar[r]&\Hom_\A(-,\Omega A)\ar[r]&\Hom_\A(-,P)\ar[r]&\Hom_\A(-,A)\ar[r]&\underline{\Hom}_\A(-,A)\ar[r]&0.}
\end{displaymath}
Therefore, $\LD_2((-)^0)(\underline{\Hom}_\A(-,A))$ is the homology of 
\begin{displaymath}
\xymatrix{0\ar[r]&(\Hom_\A(-,\Omega A))^0\ar[r]&(\Hom_\A(-,P))^0}
\end{displaymath}
which, by \thref{0hom}, is 
\begin{displaymath}
(\Hom_\A(-,\Omega A))^0=\underline{\Hom}_\A(-,\Omega A).
\end{displaymath}
\end{proof}
\begin{rmk}\thref{1hom} explains why syzygies form a functor $\Omega:\A\to\underline{\A}$.
\end{rmk}
\begin{prop}\emph{\cite[Section 4.1.1]{dean3}}\thlabel{prehereqher} If $\A$ has enough projectives, then the defect recollement
\begin{displaymath}
\xymatrix{\fp_0(\A^\op,\Ab)\ar[rr]|{\subseteq}&&\fp(\A^\op,\Ab)\ar@<2.5ex>[ll]^{(-)_0}\ar@<-2.5ex>[ll]_{(-)^0}\ar[rr]|{w}&&\mathcal{A}\ar@<2.5ex>[ll]^{Y}\ar@<-2.5ex>[ll]_{\LD_0(Y)}}
\end{displaymath}
is pre-hereditary if and only if $\A$ is hereditary.
\end{prop}
\begin{proof}
By \thref{1hom}, the recollement is pre-hereditary if and only if $\underline{\Hom}_\A(-,\Omega A)=0$ for every $A\in\A$, which is to say that every first syzygy is projective.
\end{proof}
\begin{prop}\thlabel{zeroex}Suppose $\A$ is hereditary.  
\begin{enumerate}
\item The functor $\A\to\fp_0(\A^\op,\Ab):A\mapsto\underline{\Hom}(-,A)$ is left exact. Equivalently, $\underline{\A}$ has kernels and the quotient functor $\A\to\underline{\A}$ is left exact.

\item The functor $(-)^0:\fp(\A^\op,\Ab)\to\fp_0(\A^\op,\Ab)$ is exact.
\end{enumerate}

\end{prop}
\begin{proof}\begin{enumerate}
\item Let
\begin{displaymath}
\xymatrix{0\ar[r]&A\ar[r]&B\ar[r]&C\ar[r]&0}
\end{displaymath}
be a short exact sequence. This induces an exact sequence
\begin{displaymath}
\xymatrix{\Ext^1_\A(C,-)\ar[r]&\Ext^1_\A(B,-)\ar[r]&\Ext^1_\A(A,-)\ar[r]&\Ext^2_\A(C,-)\ar[r]&\dots}
\end{displaymath}
in $\fp_0(\A,\Ab)$. Since $\Ext^2_\A(C,-)=0$, applying the duality $W$ gives an exact sequence
\begin{displaymath}
\xymatrix{0\ar[r]&\underline{\Hom}_\A(-,A)\ar[r]&\underline{\Hom}_\A(-,B)\ar[r]&\underline{\Hom}_\A(-,C).}
\end{displaymath}
\item Let $F\in\fp(\A^\op,\Ab)$ be given, and find an exact sequence
\begin{displaymath}
\xymatrix{0\ar[r]&\Hom_\A(-,A)\ar[r]&\Hom_\A(-,B)\ar[r]&\Hom_\A(-,C)\ar[r]&F\ar[r]&0}
\end{displaymath}
for some $A,B,C\in\A$. Then $\LD_1((-)^0)(F)$ is the homology of 
\begin{displaymath}
\xymatrix{0\ar[r]&(\Hom_\A(-,A))^0\ar[r]&(\Hom_\A(-,B))^0\ar[r]&(\Hom_\A(-,C))^0\ar[r]&0}
\end{displaymath}
at $(\Hom_\A(-,B))^0$, i.e. the homology of the (left exact) sequence
\begin{displaymath}
\xymatrix{0\ar[r]&\underline{\Hom}_\A(-,A)\ar[r]&\underline{\Hom}_\A(-,B)\ar[r]&\underline{\Hom}_\A(-,C)\ar[r]&0}
\end{displaymath}
at $\underline{\Hom}_\A(-,B)$, which is zero. Therefore, $\LD_1((-)^0)=0$, so $(-)^0$ is left exact. 
\end{enumerate}
\end{proof}
\begin{thm}\thlabel{mvchar}For the defect recollement
\begin{displaymath}
\xymatrix{\fp_0(\A^\op,\Ab)\ar[rr]|{\subseteq}&&\fp(\A^\op,\Ab)\ar@<2.5ex>[ll]^{(-)_0}\ar@<-2.5ex>[ll]_{(-)^0}\ar[rr]|{w}&&\mathcal{A}\ar@<2.5ex>[ll]^{Y}\ar@<-2.5ex>[ll]_{\LD_0(Y)}}
\end{displaymath}
the following are equivalent.
\begin{enumerate}
\item It is pre-hereditary.
\item It is an instance of the MacPherson-Vilonen construction.
\item $\A$ is hereditary.
\end{enumerate}
When the defect recollement is MacPherson-Vilonen, it is the MacPherson-Vilonen construction for the natural transformation $0\to \underline{Y}$, where $\underline{Y}:\A\to\fp(\A^\op,\Ab)$ is defined by $\underline{Y}A=\underline{\Hom}_\A(-,A)$ for any $A\in\A$.
\end{thm}
\begin{proof}
The equivalence of items 1 and 3 is given by \thref{prehereqher}. If item 1 holds then item 2 holds by \thref{mvtest} and \thref{zeroex} because, in the notation of \thref{mvtest}, we can set $r=(-)^0$. Also, item 2 implies item 1 by \thref{mvtest}. 

Suppose the defect recollement is an instance of the MacPherson-Vilonen construction. Consider the norm $N:\LD_0(Y)\to Y$. Using \thref{mvtest}, it is the MacPherson-Vilonen construction for $(-)^0\circ N$. Since $(-)^0\circ \LD_0(Y)=0$ and $(-)^0\circ Y\cong \underline{Y}$ by \thref{zeroex}, the defect recollement is the MacPherson-Vilonen construction for $0\to\underline{Y}$.
\end{proof}
\begin{thm}\thlabel{quot}If $\A$ is hereditary then
\begin{displaymath}
\frac{\fp(\A^\op,\Ab)}{\Ker((-)^0)}\simeq\fp_0(\A^\op,\Ab).\end{displaymath}
\end{thm}
\begin{proof}
If $\A$ is hereditary, then by \thref{zeroex}, the functor $(-)^0$ is the exact left adjoint of the inclusion functor $\fp_0(\A^\op,\Ab)\to\fp(\A^\op,\Ab)$, so it is a localisation. In particular, $(-)^0$ is a Serre localisation, and we obtain the given formula.
\end{proof}
\section{!}
When $\A$ is hereditary, the functor $(-)^0$ is a localisation. We are motivated by \thref{quot} to describe its kernel. When $\A$ is hereditary, $\fp(\A^\op,\Ab)/\fp_0(\A^\op,\Ab)\simeq \A$ and $\fp(\A^\op,\Ab)/\Ker((-)^0)\simeq \fp_0(\A^\op,\Ab)$, so it is tempting to attempt to use a ``calculus of quotients" to deduce that $\Ker((-)^0)\simeq \A$. We will prove rigorously that this result holds.

Our strategy is to characterise the subcategory of $\fp(\A^\op,\Ab)$ which consists of those finitely presented functors $F:\A^\op\to\Ab$ such that $F^0=0$ and $\LD_1((-)^0)F=0$. When $\A$ is hereditary, $\LD_1((-)^0)=0$, so this amounts to characterising the kernel of $(-)^0$ in that case.
\begin{prop}\emph{\cite[Proposition 4.11]{comprecoll}}\thlabel{factorial} In the notation of \thref{recoldef}, for a recollement situation with enough projectives, if
\begin{displaymath}
\C_!\defeq\{C\in\C:i^*(C)=0\text{ and }\LD_1(i^*)(C)= 0\}
\end{displaymath}
then $\C_!$ is the image of the fully faithful functor $j_!:\C''\to\C$, and therefore $j_!$ induces an equivalence of categories $\C''\to \C_!$.
\end{prop}
\begin{proof}
By \thref{counitstuff}, $i^*j_!=0$ and $\LD_1(i^*)j_!=0$, so $\Img(j_!)\subseteq\C_!$. Also by \thref{counitstuff}, for any $C\in\C$, $C\in\C_!$ if and only if the component $j_!j^*C\to C$ of the counit $j_!j^*\to 1_\C$ is an isomorphism, so $\Img(j_!)\supseteq\C_!$.
\end{proof}
\begin{thm}\thlabel{gllemma}In the notation of \thref{recoldef}, for any recollement situation with enough projectives,
\begin{displaymath}
\C_!=\{C\in\C:\Hom_\C(C,-)|_{\Img(i_*)}=0\text{ and }\Ext^1_\C(C,-)|_{\Img(i_*)}=0\}.
\end{displaymath}
\end{thm}
\begin{proof}
For any $C\in\C$, let 
\begin{displaymath}
\xymatrix{0\ar[r]&\Omega C\ar[r]&P\ar[r]&C\ar[r]&0}
\end{displaymath}
be an exact sequence where $P$ is projective. For any $C'\in\C'$, this induces a commutative diagram
\begin{displaymath}
\xymatrix{\Hom_\C(C,i_*C')\ar[r]\ar[d]_\cong
&\Hom_\C(P,i_*C')\ar[r]\ar[d]^\cong\ar[r]&\Hom_\C(\Omega C,i_*C')\ar[r]\ar[d]^\cong\ar[r]&\Ext^1_\C(C,i_*C')
\\
\Hom_{\C'}(i^*C,C')\ar[r]&\Hom_{\C'}(i^*P,C')\ar[r]&\Hom_{\C'}(i^*\Omega C,C')&}
\end{displaymath}
with exact rows, which shows that the assumption
\begin{align*}
\Hom_\C(C,-)|_{\Img(i_*)}&=0\\
\Ext^1_\C(C,-)|_{\Img(i_*)}&=0
\end{align*}
is equivalent to the requirement that the morphism $i^*\Omega C\to i^*P$ is an isomorphism, which, by the induced exact sequence
\begin{displaymath}
\xymatrix{0\ar[r]&\LD_1(i^*)(C)\ar[r]&i^*\Omega C\ar[r]&i^*P\ar[r]&i^*C\ar[r]&0,}
\end{displaymath}
is equivalent to the requirement that $i^*C=0$ and $\LD_1(i^*)(C)=0$.
\end{proof}
We define the subcategory $\fp_!(\A^\op,\Ab)\defeq\fp(\A^\op,\Ab)_!$ of $\fp(\A^\op,\Ab)$ with respect to the defect recollement. That is
\begin{displaymath}
\fp_!(\A^\op,\Ab)=\{F\in\fp(\A^\op,\Ab):F^0=0\text{ and }\LD_1((-)^0)(F)=0\}.
\end{displaymath}
\begin{thm}\thlabel{bangchar}For any $F\in\fp(\A^\op,\Ab)$, the following are equivalent.
\begin{enumerate}
\item $F\in\fp_!(\A^\op,\Ab)$.
\item $F\cong \LD_0(Y)A$ for some $A\in\A$.
\item There is an exact sequence 
\begin{displaymath}
\xymatrix{0\ar[r]&K\ar[r]^k&P_1\ar[r]^d&P_0}
\end{displaymath}
in $\A$ which gives a projective presentation 
\begin{displaymath}
\xymatrix{0\ar[r]&\Hom_\A(-,K)\ar[r]^{k_*}&\Hom_\A(-,P_1)\ar[r]^{d_*}&\Hom_\A(-,P_0)\ar[r]&F\ar[r]&0}
\end{displaymath}
of $F$ such that $P_0$ and $P_1$ are projective.
\end{enumerate}
\end{thm}
\begin{proof}
Items 1 and 2 are equivalent by \thref{factorial}, which says that $\A\to\fp_!(\A^\op,\Ab):A\mapsto \LD_0(Y)A$ is an equivalence. Item 2 is equivalent to item 3 by definition of $\LD_0(Y)$.
\end{proof}
We now characterise the subcategory $\fp_!(\A^\op,\Ab)$ of $\fp(\A^\op,\Ab)$ in terms of projective presentations and orthogonality properties.
\begin{defn}For morphisms $f:A\to B\in\A$ and $g:B\to C\in\A$, we say that $\underline{f}:\underline{A}\to\underline{B}$ is a \textbf{weak kernel} of $\underline{g}:\underline{B}\to\underline{C}$ in $\underline{\A}$ when the induced sequence
\begin{displaymath}
\xymatrix{\underline{\Hom}_\A(-,A)\ar[r]^{\underline{f}_*}&\underline{\Hom}_\A(-,B)\ar[r]^{\underline{g}_*}&\underline{\Hom}_\A(-,C)}
\end{displaymath}
is exact.
\end{defn}
\begin{thm}\thlabel{bang}For any $F\in\fp(\A^\op,\Ab)$, the following are equivalent.
\begin{enumerate}
\item $F\in\fp_!(\A^\op,\Ab)$.
\item For any exact sequence 
\begin{displaymath}
\xymatrix{0\ar[r]&A\ar[r]^f&B\ar[r]^g&C}
\end{displaymath}
in $\A$ which gives a projective presentation
\begin{displaymath}
\xymatrix{0\ar[r]&\Hom_\A(-,A)\ar[r]^{f_*}&\Hom_\A(-,B)\ar[r]^{g_*}&\Hom_\A(-,C)\ar[r]&F\ar[r]&0}
\end{displaymath}
of $F$, $\underline{g}:\underline{B}\to\underline{C}$ is a split epimorphism in $\underline{\A}$ and $\underline{f}:\underline{A}\to\underline{B}$ is a weak kernel of $\underline{g}:\underline{B}\to \underline{C}$.
\item There is an exact sequence
\begin{displaymath}
\xymatrix{0\ar[r]&A\ar[r]^f&B\ar[r]^g&C}
\end{displaymath}
in $\A$ which gives a projective presentation
\begin{displaymath}
\xymatrix{0\ar[r]&\Hom_\A(-,A)\ar[r]^{f_*}&\Hom_\A(-,B)\ar[r]^{g_*}&\Hom_\A(-,C)\ar[r]&F\ar[r]&0}
\end{displaymath}
of $F$ such that $\underline{g}:\underline{B}\to\underline{C}$ is a split epimorphism in $\underline{\A}$ and $\underline{f}:\underline{A}\to\underline{B}$ is a weak kernel of $\underline{g}:\underline{B}\to \underline{C}$.
\item $(F,G)\cong\Ext^1(F,G)=0$ for any additive functor $G:\A^\op\to\Ab$ which vanishes on projectives in $\A$.
\item $(F,\Ext^1_\A(-,X))\cong\Ext^1(F,\Ext^1_\A(-,X))\cong 0$ for any $X\in\A$.
\item  $(F,G)\cong\Ext^1(F,G)=0$ for any $G\in\fp_0(\A^\op,\Ab)$.
\end{enumerate}
\end{thm}
\begin{proof}
By definition, $F\in\fp_!(\A^\op,\Ab)$ if and only if $F^0=0$ and $\LD_1((-)^0)(F)=0$. Given a presentation \begin{displaymath}
\xymatrix{\Hom_\A(-,B)\ar[r]^{g_*}&\Hom_\A(-,C)\ar[r]&F\ar[r]&0}
\end{displaymath}
of $F$, $F^0$ is the cokernel of the induced morphism 
\begin{displaymath}
\xymatrix{\underline{\Hom}_\A(-,B)\ar[r]^{\underline{g}_*}& \underline{\Hom}_\A(-,C)}
\end{displaymath}
and $\LD_1((-)^0)(F)$ is the homology of the sequence
\begin{displaymath}
\xymatrix{\underline{\Hom}_\A(-,A)\ar[r]^{\underline{f}_*}&\underline{\Hom}_\A(-,B)\ar[r]^{\underline{g}_*}& \underline{\Hom}_\A(-,C).}
\end{displaymath}
Both of these are zero if and only if $\underline{g}$ is a split epimorphism in $\underline{\A}$ and $\underline{f}$ is a weak kernel of $\underline{g}$ in $\underline{\A}$. Thus item 1 implies item 2, item 2 implies item 3, and item 3 implies item 1.

Suppose item 3 holds. Let $G:\A^\op\to\Ab$ be an additive functor which vanishes on projectives in $\A$. Since $G$ obviously vanishes on morphisms which factor through a projective, $G$ corresponds to a functor $H:\underline{\A}^\op\to\Ab$. Since the sequence
\begin{displaymath}
\xymatrix{\underline{\Hom}_\A(-,A)\ar[r]^{\underline{f}_*}&\underline{\Hom}_\A(-,B)\ar[r]^{\underline{g}_*}&\underline{\Hom}_\A(-,C)\ar[r]&0}
\end{displaymath}
is exact, viewing this as a sequence of functors $\underline{\A}^\op\to\Ab$,
\begin{displaymath}
\xymatrix{0\ar[r]&(\underline{\Hom}_\A(-,C),H)\ar[r]^{(\underline{g}_*,H)}&(\underline{\Hom}_\A(-,B),H)\ar[r]^{(\underline{f}_*,H)}&(\underline{\Hom}_\A(-,A),H)}
\end{displaymath}
is exact. By the Yoneda lemma (for additive functors $\underline{\A}^\op\to\Ab$), this means that the induced sequence
\begin{displaymath}
\xymatrix{0\ar[r]&GC\ar[r]^{Gg}\ar[r]&GB\ar[r]^{Gf}&GA}
\end{displaymath}
is exact, which is equivalent to the statement that $(F,G)\cong \Ext^1(F,G)=0$. Therefore, item 3 implies item 4.

Clearly item 4 implies item 5.

Item 5 holds if and only if
\begin{displaymath}
\xymatrix{0\ar[r]&\Ext^1(C,-)\ar[r]&\Ext^1(B,-)\ar[r]&\Ext^1(A,-)}
\end{displaymath}
is exact, which, by applying the duality $W$, is equivalent to the statement that 
\begin{displaymath}
\xymatrix{\underline{\Hom}_\A(-,A)\ar[r]^{\underline{f}_*}&\underline{\Hom}_\A(-,B)\ar[r]^{\underline{g}_*}&\underline{\Hom}_\A(-,C)\ar[r]&0}
\end{displaymath}
is exact. In turn, this is equivalent to the statement that the induced sequence
\begin{displaymath}
\xymatrix{0\ar[r]&(\underline{\Hom}_\A(-,A),G)\ar[r]&(\underline{\Hom}_\A(-,B),G)\ar[r]&(\underline{\Hom}_\A(-,C),G)}
\end{displaymath}
is exact for each $G\in\fp_0(\A^\op,\Ab)$, and by \thref{stabproj} this is equivalent to the statement that the sequence
\begin{displaymath}
\xymatrix{0\ar[r]&GC\ar[r]^{Gg}&GB\ar[r]^{Gf}&GA}
\end{displaymath}
is exact. By the Yoneda lemma, this is equivalent to the statement that 
$$(F,G)\cong\Ext^1(F,G)=0.$$
Therefore, item 5 is equivalent to item 6. 

Item 6 is equivalent to item 1 by \thref{gllemma}.
\end{proof}
\begin{defn}Let $\C$ be an abelian category with enough projectives, and let $\U,\V\subseteq\C$ be full subcategories of $\C$. We say that $(\U,\V)$ is a \textbf{perpendicular pair on $\C$} if 
\begin{displaymath}
\U=\{C\in\C:\Hom_\C(C,\V)=0\text{ and }\Ext^1_\C(C,\V)=0\}
\end{displaymath}
and 
\begin{displaymath}
\V=\{C\in\C:\Hom_\C(\U,C)=0\text{ and }\Ext^1_\C(\U,C)=0\}.
\end{displaymath}
\end{defn}
Perpendicular pairs were first introduced by Geigle and Lenzing in \cite{gl} in order to study localisating subcategories.
\begin{rmk}
In \cite{auslander1965}, Auslander showed $(\fp_0(\A^\op,\Ab),\fp_*(\A^\op,\Ab))$ is a perpendicular pair on $\fp(\A^\op,\Ab)$, where $\fp_*(\A^\op,\Ab)$ is the full subcategory of representable functors.
\end{rmk}
\begin{prop}In the notation of \thref{recoldef}, for any recollement situations with enough projectives, $(\C_!,\Img(i^*))$ is a perpendicular pair on $\C$.
\end{prop}
\begin{proof}
Easy consequence of \thref{gllemma} and the recollement situation.
\end{proof}
\begin{cor}\thlabel{injtor}$(\fp_!(\A^\op,\Ab),\fp_0(\A^\op,\Ab))$ is a perpendicular pair on $\fp(\A^\op,\Ab)$.
\end{cor}
We now prove the result which we hinted at in the beginning of this section.
\begin{cor}\thlabel{hercor}If $\A$ is hereditary then $\Img(\LD_0(Y))=\Ker((-)^0)$. In particular, if $\A$ is hereditary then $\A\simeq \Ker((-)^0)$. 
\end{cor}
\begin{proof}The functor $\A\to\fp_!(\A^\op,\Ab):A\mapsto \LD_0(Y)A$ is an equivalence. If $\A$ is hereditary then $\LD_1((-)^0)=0$ so $\Ker((-)^0)=\fp_!(\A^\op,\Ab)$.
\end{proof}
Since finitely presented functors $\A^\op\to\Ab$ correspond to left exact sequences in $\A$, part of the motivation for studying them is to read our results in the category of left exact sequences and homotopy classes of chain maps. The following is an example of this.
\begin{cor}\thlabel{morphisms}For any exact sequence 
\begin{displaymath}
\xymatrix{0\ar[r]&A\ar[r]^f&B\ar[r]^g&C}
\end{displaymath}
the following are equivalent.
\begin{enumerate}
\item $\underline{g}:\underline{B}\to\underline{C}$ is a split epimorphism in $\underline{\A}$ and $\underline{f}:\underline{A}\to \underline{B}$ is a weak kernel of $\underline{g}$ in $\underline{\A}$.
\item The complex
\begin{displaymath}
\xymatrix{0\ar[r]&A\ar[r]^f&B\ar[r]^g&C\ar[r]&0}
\end{displaymath}
is homotopic to a complex
\begin{displaymath}
\xymatrix{0\ar[r]&K\ar[r]^k&P_1\ar[r]^d&P_0\ar[r]&0}
\end{displaymath}
which is exact at $P_1$ and $K$, where $P_0$ and $P_1$ are projective.
\end{enumerate}
\begin{proof}
Item 1 says that the exact sequence induces a projective presentation of a functor in $\fp_!(\A^\op,\Ab)$, by \thref{bang}. Item 2 says that the exact sequence induces a projective presentation of a functor in the image of $\LD_0(Y)$, which is equivalent to item 1, by \thref{bangchar}.
\end{proof}
\end{cor}
\section{!\textasteriskcentered}
In this section, we will discuss the functor which sends any object $A\in\A$ to $\PP(-,A)$, show that it is fully faithful, and give orthogonality properties of its image. 
\begin{defn}In the notation of \thref{recoldef}, we introduce a functor $j_{!*}=\Img N$, the image of the norm $N:j_!\to j_*$.
\end{defn}
\begin{rmk}Given a recollement of abelian categories 
\begin{displaymath}
\xymatrix{\mathcal{C}'\ar[rr]|{i_*}&&\mathcal{C}\ar@<2.5ex>[ll]^{i^!}\ar@<-2.5ex>[ll]_{i^*}\ar[rr]|{j^*}&&\mathcal{C}'',\ar@<2.5ex>[ll]^{j_*}\ar@<-2.5ex>[ll]_{j_!}}\end{displaymath} let $\epsilon:j^*j_*\to 1_{\C'}$ and $\zeta:j_!j^*\to 1_\C$ be the counits, and let $\eta:1_\C\to j_*j^*$ and $\theta:1_{\C'}\to j^*j_!$ be the units. Note that $\epsilon$ and $\theta$ are isomorphisms, because $j_!$ and $j_*$ are fully faithful, and, due to the adjunctions, we have the ``triangle identities"
\begin{displaymath}
(\epsilon j^*)(j^* \eta)=1_{j^*}=(j^* \zeta)(\theta j^*).
\end{displaymath}
Recall that the norm is defined by $N=(\zeta j_*)(j_!\epsilon^{-1})$. One could also define the ``conorm" $M=(j_*\theta^{-1})(\eta j_!)$, which would be the dual concept. However, by using the triangle identities, and the naturality of $\eta$ and $\zeta$, we obtain
\begin{align*}
Nj^*&=(\zeta j_*j^*)(j_!\epsilon^{-1}j^*)
\\&=(\zeta j_*j^*)(j_!j^*\eta)
\\&=\eta\zeta
\\&=(j_*j^*\zeta)(\eta j_!j^*)
\\&=(j_*\theta^{-1}j^*)(\eta j_!j^*)
\\&=Mj^*.
\end{align*}
Finally, since every object of $\C''$ is in the image of $j^*$, 
\begin{displaymath}
N=M.
\end{displaymath}
Therefore, the norm of a recollement is self-dual. 
\end{rmk}
\begin{thm}\thlabel{factorialstar}\emph{\cite[Proposition 4.11]{comprecoll}} In the notation of \thref{recoldef}, for any recollement of abelian categories, if
\begin{displaymath}
\C_{!*}\defeq\{C\in\C:i^*C=0\text{ and }i^!C= 0\},
\end{displaymath}
then functor $j_{!*}:\C''\to\C$ induces an equivalence $\C''\to\C_{!*}$ with inverse $j^*|_{\C_{!*}}:\C_{!*}\to \C''$. In particular, $j_{!*}$ is fully faithful.
\end{thm}
\begin{proof}
Note that $j^*N$ is an isomorphism, so $j^*j_{!*}\cong j^*j_!\cong j^*j_*\cong 1_{\C''}$.

If $C\in\C$ is in the image of $j_{!*}$ then it is a quotient of an object in $\Img(j_!)$, which implies that $i^*C=0$ since $i^*j_!=0$. Dually, $i^!C=0$. Therefore, $j_{!*}$ induces a functor $\C''\to\C_{!*}$.

Now suppose $i^*C=0$ and $i^!C=0$. By \thref{counitstuff}, the component $\zeta_C:j_!j^*C\to C$ of the counit $\zeta:j_!j^*\to 1_{\C}$ is an epimorphism, and the component $\eta_C:C\to j_*j^*C$ of the unit $\eta:1_\C\to j_*j^*$ is a monomorphism. Since $Nj^*=\eta\zeta$, there is an isomorphism
\begin{displaymath}
j_{!*}j^*C\cong\Img (N_{j^*C})\cong C
\end{displaymath}
which is natural in $C$.
\end{proof}
\begin{prop}For the defect recollement, $j_{*!}A= \PP(-,A)$, for any $A\in\A$. In particular, the functor $\A\to\fp(\A^\op,\Ab):A\mapsto \PP(-,A)$ is fully faithful.
\end{prop}
\begin{proof}
The norm is a natural transformation $N:\LD_0Y\to Y$ which is the identity $\Hom_\A(-,P)\to\Hom_\A(-,P)$ for any projective $P\in\A$. Let an object $A\in\A$ be given and suppose
\begin{displaymath}
\xymatrix{P_1\ar[r]&P_0\ar[r]&A\ar[r]&0}
\end{displaymath}
is the projective presentation of $A$. Therefore, there is a commutative diagram 
\begin{displaymath}
\xymatrix{\Hom_\A(-,P_1)\ar@{=}[d]\ar[r]&\Hom_\A(-,P_0)\ar@{=}[d]\ar[r]&(\LD_0Y)A\ar[d]^{N_A}\ar[r]&0\\
\Hom_\A(-,P_1)\ar[r]&\Hom_\A(-,P_0)\ar[r]&\Hom_\A(-,A)&}
\end{displaymath}
where the top row is exact. The image of $N_A$ is the image of the morphism $\Hom_\A(-,P_0)\to\Hom_\A(-,A)$, which is $\PP(-,A)$, as required.
\end{proof}
We define $\fp_{!*}(\A^\op,\Ab)=(\fp(\A^\op,\Ab))_{!*}$ with respect to the defect recollement. That is
\begin{displaymath}
\fp_{!*}(\A^\op,\Ab)=\{F\in\fp(\A^\op,\Ab):\exists A\in\A\text{ such that }F\cong\PP(-,A)\}.
\end{displaymath}
\begin{cor}\begin{displaymath}
\fp_{!*}(\A^\op,\Ab)=\{F\in\fp(\A^\op,\Ab):(\fp_0(\A^\op,\Ab),F)=0\text{ and }(F,\fp_0(\A^\op,\Ab))=0\}.
\end{displaymath}
\end{cor}
\begin{proof}
By \thref{factorialstar}, the image of $j_{!*}$ consists of those functors $F\in\fp(\A^\op,\Ab)$ such that $F^0\cong F_0\cong 0$.

By the recollement situation, for any $F\in\fp(\A^\op,\Ab)$ and any $G\in\fp_0(\A^\op,\Ab)$,
\begin{align*}
(G,F)&\cong (G,F_0)\\(F,G)&\cong (F^0,G)
\end{align*}
so $(\fp_0(\A^\op,\Ab),F)=0$ if and only if $F_0=0$, and $(F,\fp_0(\A^\op,\Ab))=0$ if and only if $F^0=0$.
\end{proof}
\begin{prop}\thlabel{pinfp1}For any $A\in\A$ and any $F\in\fp(\A^\op,\Ab)$, there is a monomorphism 
\begin{displaymath}
FA/F_0A\to (\PP(-,A),F)
\end{displaymath}
which is natural in $A$ and $F$. For any $x\in FA$, this monomorphism sends the coset $x+F_0A\in FA/F_0A$ to $\phi_x|_{\PP(-,A)}$, where $\phi_x:\Hom_\A(-,A)\to F$ is the unique natural transformation such that $(\phi_{x})_A(1_A)=x$.

If the sequence
\begin{displaymath}
\xymatrix{F A\ar[r]& FP\ar[r]&F\Omega A,}
\end{displaymath}
induced by any exact sequence
\begin{displaymath}
\xymatrix{0\ar[r]&\Omega A\ar[r]^k&P\ar[r]^\pi&A\ar[r]&0}
\end{displaymath}
where $P$ is projective, is exact, then the monomorphism above is an isomorphism. (Note that this hypothesis doesn't depend on the choice of such a sequence, because it is equivalent to $\Ext^1(\underline{\Hom}_\A(-,A),F)=0$.)
\end{prop}
\begin{proof}
Consider the morphism 
\begin{displaymath}
FA\to (\PP(-,A),F):x\mapsto \phi_x|_{\PP(-,A)}.
\end{displaymath}
The natural transformation $\phi_x$ is defined by $(\phi_x)_B(f)=(Ff)x$ for any $B\in\A$ and any morphism $f\in\Hom_\A(B,A)$. 

For any object $B\in\A$ and any morphism $f\in\Hom_\A(B,A)$, $f\in\PP(B,A)$ if and only if $f=\pi g$ for some morphism $g\in\Hom_\A(B,P)$. Therefore, for any $x\in FA$, $\phi_x|_{\PP(-,A)}=0$ if and only if $(F\pi)x=0$, i.e. $x\in F_0A$, so the morphism $FA/F_0A\to(\PP(-,A),F):x\mapsto \phi_x|_{\PP(-,A)}$ is a monomorphism.

Suppose the given exactness condition holds. Let a natural transformation $\alpha:\PP(-,A)\to F$ be given. Then the diagram
\begin{displaymath}
\xymatrix{\PP(P,A)\ar[d]_{k^*}\ar[r]^{\alpha_P}&FP\ar[d]^{Fk}\\
\PP(\Omega A,A)\ar[r]_{\alpha_{\Omega A}}&F\Omega A}
\end{displaymath}
commutes so 
\begin{displaymath}
(Fk)\alpha_P(\pi)=\alpha_{\Omega A}(\pi k)=0
\end{displaymath}
and hence $\alpha_P(\pi)=(F\pi)x$ for some $x\in FA$, by the exactness condition. Let $f\in\PP(B,A)$ be given for some $B\in\A$. For some $g:B\to P$, $f=\pi g$. The diagram
\begin{displaymath}
\xymatrix{\PP(P,A)\ar[d]_{g^*}\ar[r]^{\alpha_P}&FP\ar[d]^{Fg}\\
\PP(B,A)\ar[r]_{\alpha_{B}}&FB}
\end{displaymath}
commutes so 
\begin{displaymath}
\alpha_B(f)=\alpha_B g^*(\pi)=(Fg)\alpha_P(\pi)=(Fg)(F\pi)x=(Ff)x.
\end{displaymath}
This shows that $\alpha=\phi_x|_{\PP(-,A)}$. Therefore, in this case, the given monomorphism is also an epimorphism.
\end{proof}
\begin{cor}
$$\fp_0(\A^\op,\Ab)=\{F\in\fp(\A^\op,\Ab):(\fp_{!*}(\A^\op,\Ab),F)=0\}.$$
\end{cor}
\begin{proof}
If $F\in\fp_0(\A^\op,\Ab)$ then $F=F_0$, so by \thref{pinfp1},
\begin{displaymath}
(\PP(-,A),F)\cong FA/F_0A=0
\end{displaymath}
for any $A\in\A$.

If $(\fp_{!*}(\A^\op,\Ab),F)=0$ then by \thref{pinfp1}, there is a monomorphism $F/F_0\to 0$, so $F=F_0\in\fp(\A^\op,\Ab)$.
\end{proof}
\section{Applications to pp formulas}
In this section, we work out what the canonical features of the defect recollement look like on the finitely presented functors which arise naturally in the model theory of modules: Those given by pp-pairs. We will give only the algebraic characterisations of the concepts. Those interested in the dictionary between the algebraic and syntactic versions of the definitions and statements are referred to \cite{prest2009}.

Set $\A=(R\Lmods)^\op$ for this section, where $R$ is a ring. Note that, for any $F\in\fp(R\Lmods,\Ab)$, $F\in\fp_0(R\Lmods,\Ab)$ if and only if $FI=0$ for any injective left $R$-module $I$. Also, $F_0$ is the largest subfunctor of $F$ in $\fp_0(R\Lmods,\Ab)$.

Let $\Ell(R)$ be the full subcategory of $\fp(R\Lmods,\Ab)$ consisting of those functors $F:R\Lmods\to\Ab$ which are part of an exact sequence
\begin{displaymath}
\xymatrix{\Hom_R(B,-)\ar[r]&\Hom_R(A,-)\ar[r]&F\ar[r]&0}
\end{displaymath}
for $A,B\in R\lmods$. Since $R\lmods$ is closed under cokernels, any such sequence can be extended to an exact sequence 
\begin{displaymath}
\xymatrix{0\ar[r]&\Hom_R(C,-)\ar[r]&\Hom_R(B,-)\ar[r]&\Hom_R(A,-)\ar[r]&F\ar[r]&0}
\end{displaymath}
where $C\in R\lmods$. It follows by Auslander's proposition, \cite[Proposition 2.1]{auslander1965}, that $\Ell(R)$ is an abelian subcategory of $\fp(R\Lmods,\Ab)$, and by the horseshoe lemma it is closed under extensions. Clearly $\Ell(R)$ has enough projectives, and each of its projectives are of the form $\Hom_R(M,-)$ for some $M\in R\lmods$. We write $U:R\Lmods\to\Ab$ for the forgetful functor. Note that, for any $n\in\mathbb{N}$, $U^n\cong \Hom_R(R^n,-)$, which is a projective in $\Ell(R)$.
\begin{defn}A \textbf{pp formula} is a functor in $\Ell(R)$ with projective dimension at most one in $\Ell(R)$.
\end{defn}
\begin{prop}\emph{\cite[Section 10.2.1]{prest2009}} For any funcor $\phi\in\Ell(R)$, the following are equivalent.
\begin{enumerate}
\item $\phi$ is a pp formula.
\item $\phi$ is isomorphic to a subfunctor of $U^n$ for some $n\in\mathbb{N}$.
\item $\phi$ is a subobject of a projective in $\Ell(R)$.
\end{enumerate}
\end{prop}
\begin{proof}
Suppose $\phi$ is a pp formula, and 
\begin{displaymath}
\xymatrix{0\ar[r]&\Hom_R(C,-)\ar[r]^{g^*}&\Hom_R(B,-)\ar[r]&\phi\ar[r]& 0}
\end{displaymath}
is an exact sequence. Consider the kernel $f:A\to B$ of $g:B\to C$. Since $B$ and $C$ are finitely presented, $A$ is finitely generated (see \cite[Lemma 1.9]{breit1970} for this). Since 
\begin{displaymath}
\xymatrix{0\ar[r]&\Hom_R(C,-)\ar[r]^{g^*}&\Hom_R(B,-)\ar[r]^{f^*}&\Hom_R(A,-)}
\end{displaymath}
is exact, $\phi$ is isomorphic to a subfunctor of $\Hom_R(A,-)$. For some $n\in\mathbb{N}$, there is an epimorphism $R^n\to A$, and therefore a monomorphism $\Hom_R(A,-)\to \Hom(R^n,-)=U^n$. Therefore, item 1 implies item 2 and item 3. 

Suppose $\phi\in\Ell(R)$ is isomorphic to a subfunctor of a projective in $\Ell(R)$, so there is a monomorphism $\phi\to\Hom_R(C,-)$ for some $C\in R\lmods$. But since $\phi\in\Ell(R)$, there is an epimorphism $\Hom_R(B,-)\to\phi$ for some $B\in R\lmods$. The composition $\Hom_R(B,-)\to\Hom_R(A,-)$ is induced by some morphism $f:A\to B$. If $g:B\to C$ is the cokernel of $f$, then $C\in R\lmods$ and we obtain an exact sequence
\begin{displaymath}
\xymatrix{0\ar[r]&\Hom_R(C,-)\ar[r]^{g^*}&\Hom_R(B,-)\ar[r]&\phi\ar[r]&0,}
\end{displaymath}
so $\phi$ has projective dimension less than or equal to one in $\Ell(R)$. Therefore, item 3 implies item 1. In particular, item 2 also implies item 1.
\end{proof}
\begin{lem}\thlabel{nicey}Suppose $R$ is left coherent. For any $F\in\Ell(R)$, $wF\in R\lmods$ and $F_0\in\Ell(R)$.
\end{lem}
\begin{proof}
Let 
\begin{displaymath}
\xymatrix{\Hom_R(B,-)\ar[r]&\Hom_R(A,-)\ar[r]&F\ar[r]&0}
\end{displaymath}
be an exact sequence where $A,B\in R\lmods$. Applying the defect, there is an exact sequence 
\begin{displaymath}
\xymatrix{0\ar[r]&wF\ar[r]&A\ar[r]&B.}
\end{displaymath}
Since $R$ is left coherent, $R\lmods$ is an abelian subcategory of $R\Lmods$, and hence $wF\in R\lmods$, so $\Hom_R(wF,-)\in\Ell(R)$. We can construct $F_0$ as the kernel of the canonical morphism $F\to \Hom_R(wF,-)$, and therefore $F_0\in\Ell(R)$ since $\Ell(R)$ is closed under kernels.
\end{proof}
\begin{cor}Suppose $R$ is a left coherent ring. For any $\phi\in\Ell(R)$, $\phi$ is a pp formula if and only if $\phi_0=0$.
\end{cor}
\begin{proof}
Since $\phi_0$ is the kernel of the canonical morphism $\phi\to\Hom_R(w\phi,-)$, if $\phi_0=0$ then $\phi$ is isomorphic to a subfunctor of $\Hom_R(w\phi,-)$, which is a projective in $\Ell(R)$.

Conversely, $\phi$ is a subfunctor of a projective, then it is a finitely presented subfunctor of a representable functor and, since $(-)_0$ is left exact, $\phi_0=0$.
\end{proof}
\begin{defn}Any subobject (in $\Ell(R)$) of a pp formula is itself a pp formula. A pp formula which is a subfunctor of a pp formula $\phi$ is called a \textbf{pp subformula of $\phi$}. For any $n\in\mathbb{N}$, a pp subformula of $U^n$ is called a \textbf{pp-$n$-formula}.
\end{defn}
\begin{defn}For any pp-$n$-formula $\phi$, there is a natural transformation $\Hom_R(C,-)\to U^n$ for some $C\in R\lmods$ whose image is $\phi$. This natural transformation corresponds to some $\overline c\in C^n$. We call such a pair $(C,\overline c)$ a \textbf{free realisation} of $\phi$. Note that $\overline c\in\phi C$ for any such pair.
\end{defn}
\begin{prop}Let $\phi\in\Ell(R)$ be a pp-$n$-formula. A pair $(C,\overline c)$ consisting of $C\in R\lmods$ and $\overline c\in C^n$ is a free realisation of $\phi$ if and only if $\phi$ is the smallest pp-$n$-formula such that $\overline c\in\phi C$.
\end{prop}
\begin{proof}If $\chi$ is a pp-$n$-formula such that $\overline c\in \chi C$, then it is easy to show that the image of the natural transformation $\Hom_R(C,-)\to U^n$ corresponding to $\overline c\in C^n$ is contained in $\chi$. Since that image is $\phi$, this completes the proof.
\end{proof}
\begin{defn}We say that $F\in\Ell(R)$ is a \textbf{pp-pair} if $F=\phi/\psi$ for a pp formula $\phi$ and a pp subformula $\psi\leqslant \phi$. We call a projective in $\Ell(R)$ a \textbf{quantifier free pp formula}.
\end{defn}
\begin{prop}Every functor in $\Ell(R)$ is isomorphic to a pp-pair $\phi/\psi$ such that $\phi$ is a quantifier free pp formula.
\end{prop}
\begin{proof}
Let $F\in\Ell(R)$ be given and let $\pi:\Hom_R(A,-)\to F$ be an epimorphism for some $A\in R\lmods$. Then 
\begin{displaymath}
F\cong \Hom_R(A,-)/\Omega F
\end{displaymath}
where $\Omega F$ is the first syzygy of $F$. Since $F$ has projective dimension at most two, its syzygy has projective dimension at most one, and is therefore a pp subformula of the quantifier free pp formula $\Hom_R(A,-)$.
\end{proof}
For any $F\in\fp(R\Lmods,\Ab)$, $F_0$ is the largest finitely presented subfunctor of $F$ which vanishes on injective left $R$-modules. 
\begin{lem}\thlabel{artlem}Suppose $R$ is an artin algebra. For any $F\in\Ell(R)$, $F^0\in\Ell(R)$.
\end{lem}
\begin{proof}
First, note that $F^0$ is the cokernel of the canonical map $\LD_0(Y)(wF)\to F$. Since $R$ is left coherent, $wF\in R\lmods$. There is a finitely presented injective left $R$-module $I$ and a monomorphism $wF\to I$, and therefore there is an epimorphism $\Hom_R(I,-)\to\LD_0(Y)(wF)$. Therefore $F^0$ is the cokernel of a morphism $\Hom_R(I,-)\to F$. Since $\Hom_R(I,-),F\in\Ell(R)$, $F^0\in\Ell(R)$.
\end{proof}
\begin{thm}\thlabel{artres}Suppose $R$ is an artin algebra. Let $\phi/\psi$ be a pp-pair where $\phi$ and $\psi$ are pp-$n$-formulas. There is a smallest pp-$n$-formula $\rho$ such that $\psi\leqslant \rho\leqslant\phi$ and $\phi I=\rho I$ for any injective left $R$-module $I$. 
\end{thm}
\begin{proof}
Let $F=\phi/\psi$. Then $F^0$ is a quotient of $F$, so there is some subfunctor $G\leqslant F$ such that the quotient $F\to F/G$ is the counit $F\to F^0$. Now, $G$ is a subfunctor of $F$, so $G=\rho/\psi$ for some functor $\rho:R\Lmods\to\Ab$ such that $\psi\leqslant \rho\leqslant \phi$. Since $G$ is the kernel of the morphism $F\to F^0$ and $F,F^0\in\Ell(R)$ by \thref{artlem}, we have $G\in\Ell(R)$. There is an exact sequence
\begin{displaymath}
\xymatrix{0\ar[r]&\psi\ar[r]&\rho\ar[r]&G\ar[r]&0}
\end{displaymath}
and $\psi,G\in\Ell(R)$. Since $\fp(R\Lmods,\Ab)$ and $\Ell(R)$ are closed under extensions, $\rho\in\Ell(R)$, and therefore $\rho$ is a pp subformula of $\phi$. Now,
\begin{displaymath}
F^0=F/G=(\phi/\psi)/(\rho/\psi)\cong \phi/\rho,
\end{displaymath}
and the counit $F\to F^0$ is isomorphic (as an object in the comma category $F\downarrow \Ell(R)$) to the quotient $\phi/\psi\to\phi/\rho$. For any injective left $R$-module $I$, $F^0I=0$, so $\rho I=\phi I$. If $\rho'$ were another pp formula such that $\psi\leqslant\rho'\leqslant \phi$ which also agrees with injectives on $\phi$, this would give a finitely presented functor $\phi/\rho'$ and an epimorphism $\phi/\psi\to \phi/\rho'$ where $\phi/\rho'\in\fp_0(R\Lmods,\Ab)$. By the universal property of the counit, this would induce a unique morphism $\alpha:\phi/\rho\to\phi/\rho'$ such that the diagram
\begin{displaymath}
\xymatrix{\phi/\psi\ar[r]\ar[rd]&\phi/\rho\ar[d]^\alpha\\&\phi/\rho'}
\end{displaymath}
commutes. It follows that $\rho\leqslant\rho'$ and $\alpha$ is the quotient map. This completes the proof.
\end{proof}
\begin{thm}\thlabel{cohringres}Suppose $R$ is a left coherent ring. Let $\phi/\psi$ be a pp-pair for left $R$-modules, where $\phi$ and $\psi$ are pp-$n$-formulas. There is a largest pp-$n$-formula $\sigma$ such that $\psi\leqslant \sigma\leqslant\phi$ and $\psi I=\sigma I$ for any injective left $R$-module $I$. Furthermore, 
\begin{displaymath}
(\phi/\psi)_0=\sigma/\psi.
\end{displaymath}
Therefore, $\phi/\psi$ is a pp formula if and only if $\psi=\sigma$.
\end{thm}
\begin{proof}
By \thref{nicey}, $\Hom_R(wF,-)\in\Ell(R)$ and $F_0\in\Ell(R)$. Now, $F_0$ is a subfunctor of $F$, so $F_0=\sigma/\psi$ for some functor $\sigma:R\Lmods\to\Ab$ such that $\psi\leqslant\sigma\leqslant\phi$. There is an exact sequence 
\begin{displaymath}
\xymatrix{0\ar[r]&\psi\ar[r]&\sigma\ar[r]&\sigma/\psi\ar[r]&0}
\end{displaymath}
and $\psi,\sigma/\psi\in\Ell(R)$. Since both $\fp(R\Lmods,\Ab)$ and $\Ell(R)$ are closed under extensions, we have $\sigma\in\Ell(R)$, and hence $\sigma$ is a pp subformula of $\phi$. 

Since $F_0$ vanishes on injectives, $\psi$ and $\sigma$ agree on injectives. If $\sigma'$ where another pp formula such that $\psi\leqslant \sigma'\leqslant \phi$ and also agrees with $\psi$ on injectives, this would give a finitely presented subfunctor $F'=\sigma'/\psi$ of $F$ which vanishes on injectives, and since $F_0$ is the largest such subfunctor, $F'\subseteq F_0$, which implies that $\sigma'\leqslant \sigma$.
\end{proof}
\begin{defn}We write $\pp^n(R)$ for the lattice of all pp-$n$-formulas. That is, $\pp^n(R)$ is the lattice of $\Ell(R)$-subobjects of $U^n$. We call $\pp^n(R)$ the \textbf{pp-$n$-lattice for left $R$-modules.}
\end{defn}
\begin{lem} \emph{\cite[Lemma 6.1]{auslander1965}} For any left $R$-module $M$, $-\otimes_RM:\Rmods R\to\Ab$ is finitely presented if and only if $M$ is finitely presented if and only if $-\otimes_R M\in\Ell(R^\op)$.
\end{lem} 
\begin{defn}The \textbf{Auslander-Gruson-Jensen duality} is a functor
\begin{displaymath}
d:(\Ell(R))^\op\to \Ell(R^\op)
\end{displaymath}
defined by 
\begin{displaymath}
(dF)M=(F,M\otimes_R-)
\end{displaymath}
for any $F\in \Ell(R)$ and $M\in \Rmods R$.
\end{defn}
\begin{thm}\emph{\cite[Section 6.2]{garkusha2001} \cite[Section 10.3]{prest2009}} For any $F\in \Ell(R)$, $dF\in\Ell(R^\op)$ and there is an isomorphism $d^2F\cong F$ which is natural in $F$. Therefore, $d:(\Ell(R))^\op\to\Ell(R^\op)$ is an equivalence of categories, with inverse $d^\op:\Ell(R^\op)\to (\Ell(R))^\op$. For any $M\in R\lmods$, and $N\in\rmods R$, there are isomorphisms
\begin{align*}
d\Hom_R(M,-)&\cong -\otimes_R M\\
d(N\otimes_R-)&\cong \Hom_R(N,-)
\end{align*}
which are natural in $M$ and $N$.
\end{thm}
The Auslander-Gruson-Jensen duality was first written down by Auslander in \cite{auslander1986} and independently by Gruson and Jensen in \cite{gj}.

Since there is a canonical isomorphism $dU^n\cong U^n$ for any $n\in\mathbb{N}$, any pp-$n$-formula $\phi\in\Ell(R)$ gives rise to an exact sequence
\begin{displaymath}
\xymatrix{0\ar[r]&\phi\ar[r] & U^n\ar[r]&U^n/\phi\ar[r]&0}
\end{displaymath}
and hence an exact sequence
\begin{displaymath}
\xymatrix{0\ar[r]&d(U^n/\phi)\ar[r]&U^n\ar[r]&d\phi\ar[r]&0.}
\end{displaymath}
Therefore, $d(U^n/\phi)$ is canonically isomorphic to a pp-$n$-formula for right $R$-modules, which we denote by $D\phi\in\pp^n(R^\op)$. It is easy to show that we have constructed a morphism of lattices
\begin{displaymath}
D:(\pp^n(R))^\op\to \pp^n(R^\op)
\end{displaymath}
which satisfies the following properties.
\begin{itemize}
\item $D\phi \leqslant D\psi\Leftrightarrow \psi\leqslant \phi$ for any $\psi,\phi\in\pp^n(R)$.
\item $D^2\phi=\phi$ for any $\phi\in\pp^n(R)$.
\item $D$ is an isomorphism of lattices with inverse $D^\op:\pp^n(R^\op)\to(\pp^n(R))^\op$.
\item \cite[Corollary 10.3.8]{prest2009} For any pp pair $\phi/\psi$ where $\phi$ and $\psi$ are pp-$n$-formulas,
\begin{displaymath}
d(\phi/\psi)\cong D\psi/D\phi.
\end{displaymath}
\end{itemize}
The isomorphism $D:(\pp^n(R))^\op\to\pp^n(R^\op)$ is called the \textbf{elementary duality} and was first discovered by Prest, in its syntactic form, in \cite{prest1988}. Our description of it is equivalent to Prest's by \cite[Corollary 10.3.6]{prest2009}.
\begin{lem}[Holds when $\A$ is any abelian category.]\thlabel{leftdef}\emph{\cite[Theorem 4]{russell}} For any $F\in\fp(\A^\op,\Ab)$ and any left exact functor $G:\A^\op\to\Ab$, there is an isomorphism
\begin{displaymath}
G(wF)\cong (F,G)
\end{displaymath}
which is natural in $F$ and $G$.
\end{lem}
\begin{proof}
Let $\eta_F:F\to\Hom_\A(-,wF)$ be the component of the counit of the adjunction $w\dashv Y$. Let $\epsilon:wY\to 1_\A$ be the counit, which is an isomorphism (in fact, we can take it to be an identity). Since $(Y\epsilon)(\eta Y)=1_Y$, $\eta Y$ is an isomorphism. Therefore, $F$ is a representable functor if and only if $\eta_F$ is an isomorphism. Now, $\eta_F$ induces a morphism 
\begin{displaymath}
(\Hom_\A(-,wF),G)\to (F,G),
\end{displaymath}
and hence a morphism
\begin{displaymath}
G(wF)\to (F,G)
\end{displaymath}
which is an isomorphism when $F$ is representable. Let \begin{displaymath}
\xymatrix{\Hom_\A(-,B)\ar[r]&\Hom_\A(-,C)\ar[r]&F\ar[r]&0}
\end{displaymath}
be an exact sequence for some $B,C\in\A$. Since $G$ is left exact, this induces a commutative diagram
\begin{displaymath}
\xymatrix{0\ar[r]&G(wF)\ar[d]\ar[r]&G(w\Hom_\A(-,C))\ar[d]_\cong\ar[r]&G(w\Hom_\A(-,B))\ar[d]_\cong
\\
0\ar[r]&(F,G)\ar[r]&(\Hom_\A(-,C),G)\ar[r]&(\Hom_\A(-,B),G)}
\end{displaymath}
with exact rows, where the vertical arrows are given by the morphism we have constructed. Since the two vertical arrows at representables are isomorphisms, it follows that our morphism 
\begin{displaymath}
G(wF)\to (F,G)
\end{displaymath}
is also an isomorphism, as required.

\end{proof}
\begin{prop}\thlabel{wrath}For any $F\in\Ell(R)$ and any flat right $R$-module $M$, there is an isomorphism of left $\mathrm{End}_RM$-modules,
\begin{displaymath}
M\otimes_R wF\cong (dF)M.
\end{displaymath}
which is natural in $F$ and $M$. In  particular, there is an isomorphism
\begin{displaymath}
wF\cong (dF)R
\end{displaymath}
of left $R$-modules which is natural in $F$.
\end{prop}
\begin{proof}
Since $(dF)M=(F,M\otimes_R-)$ and $M\otimes_R-$ is exact, this follows by \thref{leftdef}. The second part follows by setting $M=R$.
\end{proof}
\begin{cor}\thlabel{kookoo}For any pp-pair $\phi/\psi\in\Ell(R)$, there is an isomorphism of left $R$-modules
\begin{displaymath}
w(\phi/\psi)\cong (D\psi)R/(D\phi)R.
\end{displaymath}
\end{cor}
\begin{proof}
This follows directly from the fact that $d(\phi/\psi)\cong D\psi/D\phi$ and \thref{wrath}.
\end{proof}
\begin{cor}\thlabel{calculus}For any pp-pair $\phi/\psi$ for left $R$-modules, $\phi R=\psi R$ if and only if $D\phi I=D\psi I$ for any injective $I\in\Rmods R$.
\end{cor}
\begin{proof}By \thref{kookoo},
\begin{displaymath}
w(D\psi/D\phi)\cong (D^2\phi)R/(D^2\psi)R=\phi R/\psi R,
\end{displaymath}
and $D\psi/D\phi$ vanishes on injectives if and only if $w(D\psi/D\phi)=0$. This completes the proof.
\end{proof}
Now we are in a position to dualise our results in the previous section.
\begin{thm}Suppose $R$ is an artin algebra. Let $\phi/\psi$ be a pp-pair where $\phi$ and $\psi$ are pp-$n$-formulas. There is a largest pp-$n$-formula $\mu$ such that $\psi\leqslant \mu\leqslant\phi$ and $\psi R=\mu R$.
\end{thm}
\begin{proof}
Follows from \thref{calculus}, by applying \thref{artres} to $D\psi/D\phi$.
\end{proof}
\begin{thm}Suppose $R$ is a left coherent ring. Let $\phi/\psi$ be a pp-pair for left $R$-modules, where $\phi$ and $\psi$ are pp-$n$-formulas. There is a smallest pp-$n$-formula $\nu$ such that $\psi\leqslant \nu\leqslant\phi$ and $\phi R=\nu R$. 
\end{thm}
\begin{proof}
Follows from \thref{calculus}, by applying \thref{cohringres} to $D\psi/D\phi$.
\end{proof}
Recall that, for any pp-$n$-formula $\phi\in\pp^n(R)$, $D\phi$ is the kernel of a canonical epimorphism $\pi:U^n\to d\phi$ obtained by applying the Auslander-Gruson-Jensen duality to the inclusion $\phi\to U^n$. One can show by diagram chasing that, for any right $R$-module $M$, the epimorphism
\begin{displaymath}
\pi_M:M^n\to (\phi,M\otimes_R-)
\end{displaymath}
is defined by $(\pi_M(\overline a))_N(\overline b)=\overline a\otimes_R\overline b$, for any left $R$-module $N$, where, for tuples $\overline a=(a_1,a_2,\dots,a_n)\in M^n$ and $\overline b=(b_1,b_2,\dots,b_n)\in \phi N$,
\begin{displaymath}
\overline a\otimes_R\overline b\defeq \sum^n_{i=1}a_i\otimes_R b_i.
\end{displaymath}
Therefore, since 
$$(D\phi)M=\Ker(\pi_M)=\{\overline a\in M^n:\overline a\otimes_R\phi N=0\},$$
we obtain the following.
\begin{prop}\thlabel{crazytimes}For any pp-$n$-formula $\phi\in\pp^n(R)$, with free realisation $(C,\overline c)$ and any right $R$-module $M$,
\begin{displaymath}
(D\phi)M=\{\overline a\in M^n:\overline a\otimes_R\overline c=0\}.
\end{displaymath}
\end{prop}
\begin{proof}
We know that $\overline a\in(D\phi)M$ if and only if, for any left $R$-module $N$, $\overline a\otimes_R\phi N=0$. But since the natural transformation $\Hom_R(C,-)\to\phi$ corresponding to $\overline c$ is an epimorphism, for any $\overline b\in \phi N$, there is a morphism $f:C\to N$ such that $f\overline c=\overline b$, and therefore $\overline a\otimes_R\overline b=(1_M\otimes_R f)(\overline a\otimes_R\overline c)$, so the condition $\overline a\otimes_R\overline c=0$ is necessary and sufficient for $\overline a\in (D\phi)M$.
\end{proof}
\begin{defn}Let $n\in\mathbb{N}$ be given. For any $C\in R\Lmods$ and $\overline c\in C^n$, we write define the left $R$-module
\begin{displaymath}
\Ann(C,\overline c)=\{\overline r\in R^n:\overline r\cdot \overline c=0\}.
\end{displaymath}
\end{defn}
\begin{cor}\thlabel{jambags}For any pp-$n$-formula $\phi$ with free realisation $(C,\overline c)$,
\begin{displaymath}
(D\phi)R=\Ann(C,\overline c).
\end{displaymath}
\end{cor}
\begin{proof}
The morphism $R^n\otimes_R C\to C^n:\overline r\otimes_R \overline a\mapsto \overline r\cdot\overline a$ is an isomorphism, so this follows directly from \thref{crazytimes}.
\end{proof}
\begin{prop}\thlabel{prevs}For any pp-pair $\phi/\psi\in\Ell(R)$, where $\phi$ and $\psi$ are pp-$n$-formulas, with free realisations $(C,\overline c)$ and $(D,\overline d)$ respectively,
\begin{displaymath}
w(\phi/\psi)\cong \Ann(D,\overline d)/\Ann(C,\overline c).
\end{displaymath}
\end{prop}
\begin{proof}
This follows directly from \thref{kookoo} and \thref{jambags}.
\end{proof}
\begin{prop}Let $n\in\mathbb{N}$ be given, and let $\phi/\psi\in\Ell(R)$ be a pp-pair where $\phi$ and $\psi$ are pp-$n$-formulas. Let free realisations $(C,\overline c)$ and $(D,\overline d)$ of $\phi$ and $\psi$ be given, respectively. Then 
\begin{displaymath}
w(\phi/\psi)\cong \{\overline r\cdot \overline c:\overline r\in R^n\text{ such that }\overline r\cdot\overline d=0\}.
\end{displaymath}
\end{prop}
\begin{proof}
There is a morphism such that $C\to D:\overline c\to\overline d$, and therefore, there is an epimorphism $(\overline c)\to(\overline d)$, where $(\overline c)$ is the submodule of $C$ generated by $\overline c$ and $(\overline d)$ is the submodule of $D$ generated by $\overline d$. By \thref{prevs}, there is a commutative diagram
\begin{displaymath}
\xymatrix{&0\ar[d]&0\ar[d]&0\ar[d]&\\
0\ar[r]&0\ar[r]\ar[d]&0\ar[r]\ar[d]&w(\phi/\psi)\ar[d]&
\\
0\ar[r]&\Ann(D,\overline d)\ar[d]\ar[r]&\Ann(C,\overline c)\ar[d]\ar[r]&w(\phi/\psi)\ar[d]\ar[r]&0
\\0\ar[r]&R^n\ar[d]\ar@{=}[r]&R^n\ar[d]\ar[r]&0\ar[d]\ar[r]&0
\\&(\overline c)\ar[d]\ar[r]&(\overline d)\ar[d]\ar[r]&0\ar[d]\ar[r]&0
\\&0&0&0,&}
\end{displaymath}
with exact rows and columns, and therefore, by the snake lemma, $w(\phi/\psi)$ is the kernel of the morpshism $(\overline c)\to(\overline d):\overline c\mapsto\overline d$, as required.
\end{proof}


\begin{thebibliography}{99}
\bibitem{auslander1965}Maurice Auslander, \textit{Coherent functors}. Proceedings of the Conference of Categorical Algebra, La Jolla 1965, Springer-Verlag, 1966.
\bibitem{auslander1986}
Maurice Auslander, \emph{Isolated singularities and existence of almost split sequences}. Notes by Louise Unger, Representation Theory
II 194-242, Groups and Orders, Ottawa 1984, Lecture Notes in Mathematics 1178, Springer-Verlag, 1986.
\bibitem{breit1970}Siegfried Breitsprecher, \emph{Lokal enlich pr\"asentierbare grothendieck-kategorien.} Mathematisches Seminar, 1970.
\bibitem{dean3}Samuel Dean, Jeremy Russell, \textit{Derived recollements and generalised AR formulas}. Journal of Pure and Applied Algebra,
2018.
\bibitem{comprecoll}Vincent Franjou, Teimuraz Pirashvili. \emph{Comparison of abelian categories recollements}. Documenta Mathematica 9, 41-56, 2004.
\bibitem{garkusha2001} Grigory Garkusha, \emph{Grothendieck categories}, Algebra i Analiz 13(2), 1-68, 2001 (Russian). An English translation appears in St. Petersburg Math. J. 13(2), 149-200, 2002.
\bibitem{gl}Werner Geigle, Helmut Lenzing, \emph{Perpendicular categories with applications to representations and sheaves}. Journal of Algebra 144, 273-343, 1991.
\bibitem{gj}
Laurent Gruson and Christian U. Jensen, \textit{Dimensions cohomologiques reli\'ees aux foncteurs $\underleftarrow{\lim}^{(i)}$}. Lecture Notes in Mathematics 867,Springer-Verlag, 1981.
\bibitem{macvil}Robert MacPherson, Kari Vilonen, \textit{Elementary construction of perverse sheaves}. Inventiones Mathematicae 84, 403-435, 1986.
\bibitem{prest1988}
Mike Prest, \textit{Duality and pure-semisimple rings}. Journal of the London Mathematical Society (2) 38 (1988) 403-409, 1988.
\bibitem{prest2009}
Mike Prest, \textit{Purity, Spectra and Localisation}. Cambridge University Press, 2009.
\bibitem{russell}
Jeremy Russell, \textit{Applications of the defect of a finitely presented functor}. Journal of Algebra 465 137-169, 2016.
\end{thebibliography}
\end{document}